\documentclass[a4paper,12pt]{article}
\usepackage{amsmath,amssymb,amsthm,cite,verbatim,xcolor,tikz-cd}
\usepackage[bookmarks=true]{hyperref}
\usepackage{graphicx}
\title{Infinitely many cubic points for $X_0^+(N)$ over $\mathbb{Q}$}
\author{ \ Francesc Bars\footnote{First author is supported by PID2020-116542GB-I00} and Tarun  Dalal }
\newcounter{teo}
\newtheorem{prop}[teo]{Proposition}%[section]
\newtheorem{lema}[teo]{Lemma}
\newtheorem{thm}[teo]{Theorem}%[section]
\newtheorem{cor}[teo]{Corollary}%[section]
\theoremstyle{definition}

\theoremstyle{remark}

\numberwithin{equation}{section}

\newcommand{\Q}{\mathbb{Q}}

\newcommand{\N}{\mathbb{N}}

\newcommand{\F}{\mathbb{F}}

\date{}
\begin{document}

\maketitle

\begin{abstract}
We determine all modular curves $X_0^+(N)$ that admit infinitely
many cubic points over the rational field $\mathbb{Q}$.
\end{abstract}

\section{Introduction}

A non-singular smooth curve $C$ over a number field $K$ of genus
${g_C>1}$ always has a finite set of $K$-rational points $C(K)$ by a
celebrated result of Faltings (here we fix once and for all
$\overline{K}$, an algebraic closure of $K$). We consider the set of
all points of degree at most $d$ for $C$ by
$\Gamma_d(C,K)=\cup_{[L:K]\leq d}C(L)$ and exact degree $d$ by
$\Gamma_d'(C,K)=\cup_{[L:K]=d}C(L)$, where $L\subseteq \overline{K}$
runs over the finite extensions of $K$. A point $P\in C$ is said to
be a point of degree $d$ over $K$  if $[K(P):K]=d.$

The set $\Gamma_d(C,M)$ is infinite for a certain $M/K$ finite
extension if $C$ admits a degree at most $d$ map, {{all
defined over $M$}}, to a projective line or an elliptic curve
{ with positive $M$-rank}. The converse is true for
$d=2$ \cite{HaSi91}, $d=3$ \cite{AH91} and $d=4$ under certain
restrictions \cite{AH91}\cite{DF93}. If we fix the number field $M$
in the above results (i.e. an arithmetic statement for
$\Gamma_d(C,M)$ with $M$ fix), {{we need a precise
understanding over $M$}} of the set $W_d(C)=\{v\in
Pic^d(C)|h^0(C,\mathcal{L}_v)>0\}$ where $Pic^d$ is the usual
$d$-Picard group and $\mathcal{L}_v$ the line bundle of degree $d$
on $C$ associated to $v$. {{If $W_d(C)$} contains no
translates of abelian subvarieties with positive $M$-rank of
$Pic^d(C)$ then $\Gamma_d'(C,M)$ is finite, (under the assumption
that $C$ admits no maps of degree at most $d$ to a projective line
over $M$)}.

For $d=2$ the arithmetic statement {{for
$\Gamma_d(C,K)$}} follows from \cite{AH91}, (for a sketch of the
proof and the precise statement see \cite[Theorem 2.14]{BaMomose}).

For $d=3$, Daeyeol Jeon introduced an arithmetic statement and its
proof in \cite{Jeo21} following \cite{AH91} and \cite{DF93}. In
particular if $g_C\geq 3$ and $C$ has no degree 3 or 2 map to a
projective line and no degree 2 map to an elliptic curve over
$\overline{K}$  then the set of exact cubic points of $C$ over $K$,
$\Gamma_3'(C,K)$, is an infinite set if and only if $C$ admits a
degree three map to an elliptic curve over $K$ with positive
$K$-rank.

Observe that if {{$g_C\leq 1$ (with $C(K)\neq \emptyset$
for $g_C=1$)}}, then $C$ has a degree three map over $K$ to the
projective line, thus $\Gamma_3'(C,K)$ is always an infinite set.
Thus for curves $C$ with $C(K)\neq \emptyset$ we restrict to
$g_{C}\geq 2$ in order to study the finiteness of the set
$\Gamma_d'(C,K)$. { The values of $N$ for which $X_0^+(N)$ has genus $0$ and $1$ are listed in Theorem \ref{known values of N}.}

Jeon in \cite{Jeo21} determined the finite set of modular curves
$X_0(N)$ where $\Gamma_3'(X_0(N),\mathbb{Q})$ is infinite.

In this paper, we deal with determining such question for $X_0^+(N)$
corresponding to the modular curve $X_0(N)$ quotient by the
Atkin-Lehner involution $w_N$, recall that it is an algebraic curve
defined over $\mathbb{Q}$ and $X_0^+(N)(\mathbb{Q})\neq\emptyset$,
because it has a rational cusp. The novelty with respect to previous works on degree $2$ and $3$ maps to an elliptic curve $E$ is by considering the cover $\Q(X_0(N))/\Q(E)$ taking into account the action of Atkin-Lehner involution.

The main theorem of the article is the following.

\begin{thm} Suppose $g_{X_0^+(N)}\geq 2$. Then
$\Gamma_3'(X_0^+(N),\mathbb{Q})$ is infinite if and only if
$g_{X_0^+(N)}=2$ or $N$ is in the following list:

\begin{center}
\begin{tabular}{|c|l|}
\hline $g_{X_0^+(N)}$&$N$\\
\hline 3&58,76,86,96,97,99,100,109,113,127,128,139,149,151,169,179,239 \\

\hline 4&88,92,93,115,116,129,137,155,159,215\\
\hline 5&122,146,181, 185, 227\\
\hline 6&124,163, 164, 269\\
\hline 7& 196, 243\\ \hline \hline 10&236\\
\hline
\end{tabular}
\end{center}

\end{thm}

All computations sources used in the paper are available at
\begin{verbatim}
 https://github.com/Tarundalalmath/X_0-N-with-infinitely-many-cubic-points
\end{verbatim}
except the ones for counting points over finite fields, where we use
modified versions for $X_0^+(N)$ of the ones already available at
different links in \cite{BaGon}.

\section{General considerations}

Given a complete curve $C$ over $K$, the gonality of $C$ is defined as follows:
$$\rm{Gon}(C):=min\{\deg(\varphi)\mid \varphi: C\rightarrow \mathbb{P}^1 \mathrm{defined \ over} \ \overline{K}\}.$$
%the
%minimal degree map to a projective line in $\overline{K}$ and such
%natural number is denoted by $\rm{Gon}(C)$.

By~\cite[Lemma 1.2]{Jeo21} (and arguments there):

{
\begin{lema}\label{N not in known set the it has map to elliptic curve}
Suppose $C$ has $Gon(C)\geq 4$, { $P\in C(K)$} and does not have a degree $\leq 2$ map to
an elliptic curve. If $\Gamma_3'(C,K)$ is an infinite set then $C$ admits a $K$-rational map of degree $3$ to an elliptic
curve with positive $K$-rank.
%
%{\color{red} Then the followings are equivalent
%\begin{enumerate}
%\item $\Gamma_3'(C,K)$ is an infinite set,
%\item $W_3(C)$ contains a
%translation of an elliptic curve with positive $K$-rank,
%\item $C$ admits a $K$-rational map of degree $3$ to an elliptic
%curve with positive $K$-rank.
%\end{enumerate}
% }
% \begin{proof}
% Although the proof is exactly similar to \cite[Lemma 1.2]{Jeo21}, for completeness we are giving the proof.
% Consider the mapping $\iota : C^{(3)}\rightarrow Jac(C)$ defined by, $\iota(P_1,P_2,P_3)=[P_1+P_2+P_3-3P]$, where $Jac(C)$ denotes the Jacobian of the curve $C$. Since $P\in C(K)$, the mapping $\iota$ is defined over $K$. Let $W_3(C):=\iota(C^{(3)})$. Since $\rm(Gon)(C)\geq 4$, the mapping $\iota$ is injective and we have $C^{(3)}$
% \end{proof}
%
%Then $\Gamma_3'(C,K)$ is an infinite set if and
%only if $C$ admits a $K$-rational map of degree $3$ to an elliptic
%curve with positive $K$-rank, if and only if $W_3(C)$ contains a
%translation of an elliptic curve with positive $K$-rank.
\end{lema}}

{The modular curves $X_0^+(N)$ for which Lemma
\ref{N not in known set the it has map to elliptic curve} is not applicable are quoted
in the next result}, corresponding to the works \cite{MaHa99},
\cite{D18} and \cite{HS99} (the list with $g_{X_0^+(N)}\leq 1$ is
well-known and follows easily from \cite[Appendix]{BaGon}).

\begin{thm}\label{known values of N}
\begin{enumerate}
\item The modular curve $X_0^+(N)$ has $g_{X_0^+(N)}=0$ if and only if $N$ is one of the following:

1-21, 23-27, 29, 31, 32, 35, 36, 39, 41, 47, 49, 50, 59, 71.

\item $X_0^+(N)$ is an elliptic curve (equivalently $g_{X_0^+(N)}=1$) if and only if $N$ is one of the following:

22, 28, 30, 33, 34, 37, 38, 40, 43, 44, 45, 48, 51, 53-56, 61,
63-65, 75, 79, 81, 83, 89, 95, 101, 119, 131.

\item (Furumoto-Hasegawa) $X_0^+(N)$ is hyperelliptic if and only if $N$ is one of the following:

42, 46, 52, 57, 60, 62, 66-69, 72-74, 77, 80, 85, 87, 91, 92, 94,
98, 103, 104, 107, 111, 121, 125, 143, 167, 191.

\item (Jeon) $X_0^+(N)$ is bielliptic,i.e. has a degree two map to an elliptic curve, if and only if $N$ is one of the following:

42, 52, 57, 58, 60, 66, 68, 70, 72, 74, 76-78, 80, 82, 84-86, 88,
90, 91, 96, 98-100, 104, 105, 108, 110, 111, 117, 118, 120, 121,
123, 124, 128, 135, 136, 141-145, 155, 159, 171, 176, 188.

\item (Hasegawa-Shimura) $\mathrm{Gon}(X_0^+(N))=3$ if and only if $N$ is one of the following:

%58, 76, 86, 96, 97, 99, 100, 109, 113, 127, 128, 139, 149, 151, 169,
%179, 239, 70, 82, 84, 88, 90, 93, 108, 115, 116, 117, 129, 135, 137,
%147, 155, 159, 161, 173, 199, 215, 251, 311, 122, 146, 181, 227,
%164, 162.
58, 70, 76, 82, 84, 86, 88, 90, 93, 96, 97, 99, 100, 108, 109, 113, 115, 116, 117, 122, 127, 128, 129, 135, 137, 139, 146, 147, 149, 151, 155, 159, 161, 162, 164, 169, 173, 179, 181, 199, 215, 227, 239, 251, 311.

\end{enumerate}
\end{thm}

We say that a pair $(N,E)$ where $N$ is a natural number and $E$ is an
elliptic curve over $\mathbb{Q}$ with positive $\mathbb{Q}$-rank is
admissible if there is a degree 3 map over $\mathbb{Q}$ of the form
$X_0^+(N)\rightarrow E$.

{ The following lemma gives a criterion to rule out the pairs which are not admissible.}

\begin{lema}\label{sieve}
 If $(N,E)$ is an admissible pair, then:
\begin{enumerate}
\item $E$ has conductor $M$ with $M|N$ and for any prime $p\nmid N$ we have
$|\overline{X}_0^+(N)(\mathbb{F}_{p^n})|\leq 3|\overline{E}(\mathbb{F}_{p^n})|$ and
$|\overline{X}_0(N)(\mathbb{F}_{p^n})|\leq 6|\overline{E}(\mathbb{F}_{p^n})|$, $\forall n\in \N$.
%, for all $n\geq 1$ natural number.
%\item $E$ has conductor $M$ with $M|N$, and $\mathbb{Q}$-rank of $E$ is at least 1.
\item if conductor of $E$ is $N$, then the degree of the strong Weil
parametrization of $E$ divides 6.
\item \label{iii} for any prime $p\nmid N$ we have
$\frac{p-1}{12}\psi(N)+2^{\omega(N)}\leq 6 (p+1)^2$, where
$\omega(N)$ is the number of prime divisors of $N$ and {$\psi=N \prod_{\substack{q|N\\ q \ prime}}(1+1/q)$} is the
$\psi$-Dedekind function,

\item for any $w_r$ Atkin-Lehner involution of $X_0(N)$ with $r\neq
N$ we have $g_{X_0^+(N)}\leq 3+2\cdot g_{X_0^+(N)/w_r}+2$.

\end{enumerate}
\end{lema}
\begin{proof}
Let $(N,E)$ be admissible. 
Then there is a $\Q$ rational degree $3$ mapping $f:X_0^+(N)\rightarrow E$ consequently we have a $\Q$ rational degree $6$ mapping $g:X_0(N)\rightarrow E$. 
Hence $\mathrm{conductor}(E)|N$. 
\begin{enumerate}

\item Let $p\nmid N$ be a prime. Since $p\nmid N$, the curves $X_0^+(N), X_0(N)$ and $E$ has good reduction at $p$ and the mappings $f,g$ induces the $\F_p$ rational mappings 
$\overline{f}:\overline{X}_0^+(N)\rightarrow \overline{E}$ and $\overline{g}:\overline{X}_0(N)\rightarrow \overline{E}$, where $\overline{X}_0^+(N), \overline{X}_0(N)$ and $\overline{E}$ denote the mod $p$ reduction of $X_0^+(N), X_0(N)$ and $E$ respectively. Hence we have 
$|\overline{X}_0^+(N)(\mathbb{F}_{p^n})|\leq 3|\overline{E}(\mathbb{F}_{p^n})|$ and
$|\overline{X}_0(N)(\mathbb{F}_{p^n})|\leq 6|\overline{E}(\mathbb{F}_{p^n})|$, $\forall n\in \N$. 
\item If $\mathrm{conductor}(E)=N$, and $E^\prime$ denotes the strong Weil curve with strong Weil parametrization $\varphi: X_0(N) \rightarrow E^\prime$ then we have there exists an isogeny $\psi: E^\prime \rightarrow E$ such that $g=\psi \circ \varphi$, hence the degree of the strong Weil parametrization divides $6$.
\item For any prime $p\nmid N$ we know that $|\overline{X}_0(N)(\mathbb{F}_{p^2})|\geq \frac{p-1}{12}\psi(N)+2^{\omega(N)}$ (cf. \cite[Lemma 3.1]{HaS99}) and $|\overline{E}(\F_{p^2})|\leq (p+1)^2$. Hence we have $\frac{p-1}{12}\psi(N)+2^{\omega(N)}\leq 6 (p+1)^2$.
\item We have $f: X_0^+(N)\rightarrow E$ is a degree $3$ mapping. If $w_r$ is an Atkin-Lehner involution on $X_0(N)$ with $r\neq N$, then we have a degree $2$ mapping $X_0^+(N)\rightarrow X_0^+(N)/w_r$. The result follows from Castellnuovo's inequality.
\end{enumerate}
\end{proof}

As an immediate application of Lemma \ref{sieve}(\ref{iii}) we obtain the following:
\begin{cor}
For $N> 623$, the pair $(N,E)$ is not admissible.
\end{cor}
\begin{proof}
The proof is similar to \cite[Lemma 3.2]{HaS99}. We will show that for $N\geq 623$, there exists a prime $p\nmid N$ such that $\psi(N)> \frac{12}{p-1}(6(p+1)^2-2^{w(N)})$.
\begin{itemize}
\item If $2\nmid N$ and $N> 623$ choosing $p=2$, we have
$\psi(N)\geq N+1>624=12(6.(2+1)^2-2)\geq 12(6.(2+1)^2-2^{w(n)})$.
\item If $2|N$, $3\nmid N$ and $N> 376 $ choosing $p=3$, we have $\psi(N)\geq \frac{3N}{2}>564=\frac{12}{2}(6.16-2)>\frac{12}{2}(6.16-2^{w(N)})$.
\item If $2.3|N$, $5\nmid N$ and $N> 321 $ choosing $p=5$, we have $\psi(N)\geq N\cdot\frac{3}{2}\cdot\frac{4}{3}>\frac{12}{4}(6.36-2)$.
\item If $2.3.5|N$, $7\nmid N$ and $N>319$ choosing $p=7$, we have $\psi(N)\geq N\cdot\frac{3}{2}\cdot\frac{4}{3}\cdot\frac{6}{5}>\frac{12}{6}(6.64-2)$.
\item If $2.3.5.7|N$, choose $p$ to be the smallest prime does not divide $N$.
\end{itemize}
\end{proof}

%\begin{proof} The first result is an Ogg trick, see also \cite[Lemma
%3.1]{HS99}. The second follows from considering the degree three map
%for the pair $(N,E)$ or the degree 6 map by composing with the degree
%two map $X_0(N)\rightarrow X_0^+(N)$. The third is from Carayol
%statement. The fourth follows from Castellnuovo inequality. The last
%one is the characterization of the strong Weil parametrization for
%$E$.
%\end{proof}
 After applying Lemma \ref{sieve} (see appendix B for list of $N$'s that we can discard in each item), we are reduced to
 a finite set $N$. To deal with the remaining admissible pairs the next two lemmas will be helpful.

%\vspace{3cm}
%\begin{center}
%\begin{tabular}{|c|c||c|c||
%
%\end{tabular}
%\end{center}

%{\color{blue}
{\begin{lema}\label{lem1}
%{\color{red}
Let $E/\mathbb{Q}$ be an elliptic curve of conductor $N$ and $\varphi: X_0(N)\rightarrow E$ be the strong Weil parametrization of degree $k$ defined over $\mathbb{Q}$.
% Let $\varphi: X_0(N)\rightarrow E$ be a mapping of degree
%$k$ defined over $\mathbb{Q}$, with conductor of $E$ equal to
%$N$, where $\varphi$ is the strong Weil parametrization.
Suppose
that $w_N$ acts as +1 on the modular form $f_E$ associated to $E$.
Then $\varphi$ factors through $X_0^+(N)$ (and $k$ is even).
\end{lema}
\begin{proof}
Consider the mapping $\varphi: X_0(N)\rightarrow E$.
% the field extension given by ${\varphi^*}$:
%$\mathbb{Q}(E)\subset \mathbb{Q}(X_0(N))$ and we make fix by $w_N$
%that by hypothesis act in the two fields.
%%{\color{red}
Following \cite[p.424]{CaEm} (or \cite[\S 2]{Del05}), the fact
$w_N f_E=f_E$ implies $\varphi\circ w_N=\varphi+P$, where $P$ is a torsion point of $E$ given by $P=\varphi(0)-\varphi(\infty)$, where $0,\infty$ are the corresponding cusps on $X_0(N)$ with $\varphi(\infty)=O_E$ (recall that $O_E$ denotes the zero point of $E$).
%By hypothesis, the fact
%$w_N f_E=f_E$ implies $\varphi\circ w_N=\varphi+P$ $P$ a torsion
%point of $E$, that is given by $P=\varphi(0)-\varphi(\infty)$ where
%$0,\infty$ are the corresponding cusps in $X_0(N)$ following
%\cite[p.424]{CaEm} with $\varphi(\infty)$ is $O_E$, the zero point
%of $E$.
Because the sign of the functional equation of $f_E$ is -1, the
$\Q$-rank of $E$ is odd (cf. \cite[\S 3.1]{MaSD}), this implies that $P=O_E$ (see loc.cit.
\cite{CaEm}), therefore $\varphi$ factors through the quotient
$X_0(N)/\langle w_N\rangle$ and $w_N$ acts as identity
on $E$.
\end{proof}

%\begin{lema}\label{lem2}
% Let $N$ be natural number which is not a power of a prime number. Take $(N,E)$ a pair,
% with conductor of $E$ equal to $M$ with $M|N$ and $M\neq N$. Take $d$
%a natural number with $d||M$ with $(d,N/d)=1$ and assume that $w_d$ acts as
%+1 on the modular form $f_E$ associated to $E$. Suppose that $E$ has no
%trivial $2$-torsion over $\mathbb{Q}$. Then $(N,E)$ is not
%admissible.
%\end{lema}
%
%\begin{proof} Suppose $(N,E)$ is admissible, then we have a degree
%three extension $\mathbb{Q}(X_0^+(N))/\mathbb{Q}(E)$, and take
%invariant by $w_d$. This is well-defined, because take
%$\varphi:X_0^+(N)\rightarrow E$ and $\varphi^*(\operatorname{d}z)=
%constant\cdot f(z)\operatorname{d}z$ where $f$ is isogenious to
%$f_E$ but $f$ could be though as a fix modular form of $\Gamma_0(N)$
%fixed by $w_N$, that is any lift of $f_E$ from level $M$ to level $N$
%under $w_N$ acting as +1. For any such $f$, by \cite{AL} or
%\cite[Lemma 2.1, Prop.2.2]{BaGon}, satisfies that $w_d f=f$,
%therefore $\varphi\circ w_d=\varphi +P$ with $P$ a 2-torsion point
%of $E$ (see for example \cite[p.112 \S.2]{Del05}, or see the next
%lemma). Thus $\mathbb{Q}(E)/\mathbb{Q}(E)^{w_d}$ is a degree two
%extension, therefore a 2-isogeny for the elliptic curve $E$, thus a
%non-trivial $2$-torsion. Which is a contradiction.
%\end{proof}

\begin{lema}\label{descent} Consider a degree {$k$} map $\varphi:X \rightarrow E$ defined over $\Q$
where $X$ is a quotient modular curve $X_0(N)/W_N$ with $W_N$ a
proper subgroup of $B(N)$ ($B(N)$ is the subgroup of $Aut(X_0(N))$
generated by all Atkin-Lehner involutions). Assume that $cond(E)=M$ $(M|N)$. Let $d\in \mathbb{N}$ with $d||M$, $(d,N/d)=1$ and { $w_d\notin W_N$}, such that $w_d$ acts as $+1$ on the modular form $f_E$ associated to $E$. Then,
%if $E$ has $2$-torsion over
%$\Q$ and $({\color{red}k},2)=1$ then we obtain a degree {\color{red}$k$} map
%$\varphi':X/\langle w_d\rangle\rightarrow E$, corresponding taking
%$w_d$-invariant to $\varphi$, if $E$ has no $2$-torsion over $\Q$,
%then $\varphi$ factors through $X/\langle w_d\rangle$ and {\color{red}$k$} is
%even.
\begin{enumerate}
\item if $E$ has no non-trivial $2$-torsion over $\Q$,
then $\varphi$ factors through $X/\langle w_d\rangle$ and {$k$} is even.
\item if $E$ has non-trivial $2$-torsion over
$\Q$ and $k$ is odd, then we obtain a degree
{$k$} map $\varphi':X/\langle w_d\rangle\rightarrow
{E^\prime}$, by taking $w_d$-invariant to $\varphi$,
{where $E^\prime$ is an elliptic curve isogenous to
$E$}.
%\footnote{{\color{red}it is not clear to me why $\mathbb{C}/\Lambda \cong \mathbb{C}/\langle \Lambda , P\rangle$, but we can say that $\mathbb{C}/\Lambda$ and $\mathbb{C}/\langle \Lambda , P\rangle$ are isogenous}}.
\end{enumerate}
\end{lema}
\begin{proof}
{
Let $E(\mathbb{C})\cong \mathbb{C}/\Lambda$. The mapping $\varphi$ can be considered as a mapping in the complex field $\tilde{\varphi}:\Gamma\char`\\\mathbb{H}\rightarrow \mathbb{C}/\Lambda$ defined by $\tau\mapsto \int_{i\infty}^{\tau} constant\cdot f(\tau')d\tau'$, where $\Gamma:=\langle \Gamma_0(N),W_N\rangle $ and $f\in \oplus_{d|N/M} \Q f_E(q^d)\in S_2(\Gamma_0(N))^{\langle W_N \rangle}$ (cf. \cite{Go12}). Since $w_d$ acts on $f_E$ as $+1$, it also acts on $f$ as $+1$. Moreover, $\tilde{\varphi}(w_d\tau)-\tilde{\varphi}(\tau)=P$ is independent of $\tau$. Thus $\varphi\circ w_d=\varphi +P$. Since $w_d$ is an involution, we obtain $2P\in\Lambda$,
and $P$ is a 2-torsion point of $E(\mathbb{C})$ (which could be
trivial zero point of $E$, i.e. belonging to $\Lambda$). Therefore
we have the following commutative diagram ($proj$ is the usual projection map):
\begin{center}

%$$
%\begin{array}{rcccl}
%&X(\mathbb{C})&\rightarrow^{\varphi}&\mathbb{C}/\Lambda&\\
%proj&\downarrow&&\downarrow&proj\\
%&X/\langle
%w_d\rangle(\mathbb{C})&\rightarrow^{\varphi^{w_d}}&\mathbb{C}/\langle
%\Lambda, P\rangle\\
%\end{array}
%.$$

$$\begin{tikzcd}
     & X(\mathbb{C})  \arrow[d, "proj"']     \arrow[r, "{\varphi}"] &\mathbb{C}/\Lambda \arrow[d, "{proj}"] \\ & X/\langle
w_d\rangle(\mathbb{C})    \arrow[r, "{\varphi^{w_d}}"] &\mathbb{C}/\langle
\Lambda, P\rangle
    \end{tikzcd}
$$
\end{center}
Thus when $E$ has no non-trivial $2$-torsion over $\mathbb{Q}$, then $P$ is the trivial zero of $E$ and $\varphi$ factors through $X/\langle w_d\rangle$.

On the other hand, if $E$ has non-trivial $2$-torsion over $\mathbb{Q}$ and $k$ is odd, then from the above commutative diagram we have $P$ is a non-trivial $2$-torsion point of $E$ and $\varphi$ induces a $\Q$ rational degree $k$ mapping $\varphi':X/\langle w_d\rangle\rightarrow {E^\prime}$ where ${E^\prime (\mathbb{C})} \cong \mathbb{C}/\langle \Lambda, P \rangle$ and ${E^\prime}$ is isogenous to $E$.
}

\end{proof}

}

{
As an immediate corollary of Lemma \ref{descent} we obtain:
\begin{cor}\label{lem2}
 Let $N$ be natural number which is not a power of a prime number. Take $(N,E)$ a pair,
 with conductor of $E$ equal to $M$ with $M|N$ and $M\neq N$. Let $d$ be a natural number with $d||M$,$(d,N/d)=1$ such that $w_d$ acts as
$+1$ on the modular form $f_E$ associated to $E$. Suppose that $E$ has no non-trivial $2$-torsion over $\mathbb{Q}$. Then $(N,E)$ is not
admissible.
\end{cor}
\begin{proof}
If $(N,E)$ is admissible, then we have a degree three mapping $\varphi : X_0^+(N)\rightarrow E$. Since $w_d$ acts as $+1$ on $f_E$ and $E$ has no non-trivial $2$-torsion over $\mathbb{Q}$, by Lemma \ref{descent} the map $\varphi$ factors through $X_0^+(N)/\langle w_d \rangle$. This is a contradiction since $\varphi$ has degree $3$.
\end{proof}}

\section{The curve $X_0^+(N)$, with $N$ not listed in Theorem \ref{known values of N}}

Here by Lemma \ref{N not in known set the it has map to elliptic
curve} it is enough to determine the admissible pairs $(N,E)$.
 After applying Lemma \ref{sieve} (see appendix B for list of $N$'s that we can discard in each item), we are reduced to
 {the following } finite set of candidates for admissible pairs.
 % and the 2-torsion of the elliptic curve:

%\vspace{3cm}
\begin{center}
%\begin{tabular}{|c|c|c||c|c|c||}
%\hline $(N,E)$&$AL-action\; on\; E$&$(N,E)$&$AL-action\; on\; E$\\
%\hline $(106,53a)$&$w_{53}=+$&$(114,57a)$&$w_{3}=+,w_{19}=+$\\
% $(130,65a)$&$w_{5}=+,w_{13}=+$&$(158,79a)$&$w_{79}=+$\\
%  $(163,163a)$&$w_{163}=+$&$(166,83a)$&$w_{83}=+$\\
%   $(172,43a)$&$w_{43}=+$&$(174,58a)$&$w_{2}=+,w_{29}=+$\\
%    $(178,89a)$&$w_{89}=+$&$(182,91a)$&$w_{7}=+,w_{13}=+$\\
%     $(182,91b)$&$w_{7}=-,w_{13}=-$&$(183,61a)$&$w_{61}=+$\\
%      $(185,37a)$&$w_{37}=+$&$(185,185c)$&$w_{185}=+$\\
%       $(195,65a)$&$w_{5}=+,w_{13}=+$&$(196,196a)$&$w_{196}=+$\\
%        $(202,101a)$&$w_{101}=+$&$(231,77a)$&$w_{7}=+,w_{11}=+$\\
%         $(236,118a)$&$w_2=+,w_{59}=+$&$(236,236a)$&$w_{236}=+$\\
%        { $(237,79a)$}&$w_{79}=+$&$(243,243a)$&$w_{243}=+$\\
%          $(249,83a)$&$w_{83}=+$&$(258,43a)$&$w_{43}=+$\\
%           $(258,129a)$&$w_{3}=+,w_{43}=+$&{$(267,89a)$}&$w_{89}=+$\\
%            $(269,269a)$&$w_{269}=+$&
%            %$(282,E141a)$&$w_{3}=-,w_{47}=-$
%           %{\color{red} $(282,141d)$}&{\color{red}$w_{3}=+,w_{47}=+$}
%            \\
% \hline
%\end{tabular}

\begin{tabular}{|c|c|c||c|c|c||}
\hline $(N,E)$&$AL-action\; on\; E$&$(N,E)$&$AL-action\; on\; E$\\
\hline $(106,53a)$&$w_{53}=+$&$(195,65a)$&$w_{5}=+,w_{13}=+$\\
$(114,57a)$&$w_{3}=+,w_{19}=+$&$(196,196a)$&$w_{196}=+$\\
 $(130,65a)$&$w_{5}=+,w_{13}=+$&$(202,101a)$&$w_{101}=+$\\
 $(158,79a)$&$w_{79}=+$&$(231,77a)$&$w_{7}=+,w_{11}=+$\\
  $(163,163a)$&$w_{163}=+$&$(236,118a)$&$w_2=+,w_{59}=+$\\
  $(166,83a)$&$w_{83}=+$&$(236,236a)$&$w_{236}=+$\\
   $(172,43a)$&$w_{43}=+$&$(237,79a)$&$w_{79}=+$\\
   $(174,58a)$&$w_{2}=+,w_{29}=+$&$(243,243a)$&$w_{243}=+$\\
    $(178,89a)$&$w_{89}=+$&$(249,83a)$&$w_{83}=+$\\
    $(182,91a)$&$w_{7}=+,w_{13}=+$&$(258,43a)$&$w_{43}=+$\\
     $(182,91b)$&$w_{7}=-,w_{13}=-$&$(258,129a)$&$w_{3}=+,w_{43}=+$\\
     $(183,61a)$&$w_{61}=+$&{$(267,89a)$}&$w_{89}=+$\\
      $(185,37a)$&$w_{37}=+$&$(269,269a)$&$w_{269}=+$\\
      $(185,185c)$&$w_{185}=+$
            %$(282,E141a)$&$w_{3}=-,w_{47}=-$
           %{\color{red} $(282,141d)$}&{\color{red}$w_{3}=+,w_{47}=+$}
            \\
 \hline

\end{tabular}
\end{center}

{When $(N,E)$ is in the table above with $\mathrm{cond}(E)=N$, the strong Weil parametrization $X_0(N)\rightarrow E$ has
degree 6. Thus we conclude by Lemma \ref{lem1} that $(N,E)$ is an
admissible pairing. More precisely, we have

\begin{cor} For $N=163, 185, 196,236,243$ and $269$ the modular curve
$X_0^+(N)$ has infinitely many cubic points over $\mathbb{Q}$.
\end{cor}
To deal with the remaining cases we use Lemma \ref{descent}.
\begin{cor}
For $N=106, 114, 158, 166, 172, 174, 178, 182, 183, 202,
231$,{$ 237$}, $249, 258$ {and $267$},
$\Gamma_3'(X_0^+(N),\mathbb{Q})$ is a finite set.
%For all $N$ that does not appear in Theorem \ref{known values of
%N}, and $N\neq 130,185$ $195$, $163,196,236,243$ and $269$, then
%$\Gamma_3'(X_0^+(N),\mathbb{Q})$ is a finite set.
\end{cor}
\begin{proof}
Let $N$ be as in the statement and $(N,E)$ be a pair appearing in the above table. Then $\mathrm{cond}(E)|N$, $\mathrm{cond}(E)\ne N$ and $E$ has no non-trivial $2$-torsion over $\mathbb{Q}$.
By Corollary \ref{lem2}, we conclude that the pair $(N,E)$ is not admissible (for $(182,91b)$ use the Atkin-Lehner operator $w_{91}$). The result follows.
\end{proof}

\begin{prop}
The modular curves $X_0^+(130), X_0^+(195)$ have finitely many cubic
points over $\Q$.
\end{prop}
\begin{proof}
We need to check the pairs $(130,65a)$ and $(195,65a)$. Consider
$\varphi:X_0^+(130)\rightarrow 65a$ of degree 3, we know that $w_5$
and $w_{13}$ acts as $+1$.
{ Since the degree of
$\mathbb{Q}(X_0^+(130))/\mathbb{Q}(X_0^*(130))$ is
coprime to $3$ (recall that $X_0^*(N):=X_0(N)/B(N)$ where $B(N)$ is the subgroup of $Aut(X_0(N))$
generated by all Atkin-Lehner involutions), by applying Lemma \ref{descent} twice with $w_5$ and $w_{13}$
% thus taking invariant for $\langle
%w_5,w_{13}\rangle$ applying twice Lemma \ref{descent}
we obtain a degree $3$ morphism (moreover an isogeny) between the
elliptic curves $X_0^*(130)\rightarrow E^\prime$,   where $E^\prime$
is isogenous to $65a$ (note that $X^*(130)$ has genus $1$ and its
Cremona level is $65a$)}. This is a contradiction since the elliptic
curve $65a$ has no non-trivial $3$-torsion over $\mathbb{Q}$, and
also no $3$-isogeny over $\mathbb{Q}$ by \cite{Cre}. Thus
$(130,65a)$ is not admissible. A similar argument holds for the pair
$(195, 65a)$, recall that $X_0^*(195)$ has genus $1$ and its Cremona
level is $65a$.

%We need to check the pairs $(130,65a)$ and $(195,65a)$.
%Consider $X_0^+(130)\rightarrow 65a$ of degree 3, we know that $w_5$ and $w_{13}$ acts as $+1$, thus taking invariant for such group we obtain because the degree of
%$\mathbb{Q}(X_0^+(130))/\mathbb{Q}(X_0^*(130))$
%is coprime with 3, we obtain a degree 3 morphism between the elliptic curves $X_0^*(130)\rightarrow 65a$ but $65a$ has no $3$-torsion, thus is not possible, where $X_0^*(N)$ is the modular curve $X_0(N)/B(N)$
%where $B(N)$ is the subgroup generated by all the Atkin-Lehner involutions of $X_0(N)$. A similar
%argument holds for the pair $(195, 65a)$, recall $X_0^*(195)$ has genus $1$.
\end{proof}
}

%Thus all the remaining cases in the table bellow, the pair $(N,E)$
%with conductor $E$ a strict divisor of $N$ are not admissible by
%Lemma \ref{lem2}.

%\begin{cor} For all $N$ that does not appear in Theorem \ref{known values of
%N}, and $N\neq$ $130,163,185,195,196,236,243$ and $269$, then
%$\Gamma_3'(X_0^+(N),\mathbb{Q})$ is a finite set.
%\end{cor}
%{\color{red}\begin{cor} For all $N$ that does not appear in Theorem \ref{known values of
%N}, and $N\neq$ $130, 163, 185, 195, 196, 236,243$ and $269$, then
%$\Gamma_3'(X_0^+(N),\mathbb{Q})$ is a finite set.
%\end{cor}}
%
%
%\begin{proof}
%
%\end{proof}

\section{The curve $X_0^+(N)$, with $N$ listed in Theorem \ref{known values of N}}

{Recall that $X_0^+(N)(\mathbb{Q})\neq \emptyset$.
Thus $\Gamma_3'(X_0^+(N),\mathbb{Q})$ is an infinite set when $g_{X_0^+(N)}\leq 1$.}

We assume once and for all $g_{X_0^+(N)}\geq 2$.

\subsection{ The levels $N$ with $X_0^+(N)$ hyperelliptic}

We deal with the following levels $N$:

$$\begin{tabular}{c|c}
$g_{X_0^+(N)}$ & $N$\\
\hline
2 & $42, 46, 52, 57, 62, 67, 68, 69, 72, 73, 74, 77, 80, 87, 91, 98, $\\
 & 103, 107, 111,
121, 125, 143, 167, 191\\
 \hline
3 & 60, 66, 85, 104\\
\hline 4 & 92, 94
\end{tabular}$$

For such hyperelliptic curves, we pick the model given by Hasegawa
in \cite{Ha95} if $g_{X_0^+(N)}=2$, and by Furumoto and Hasegawa in
\cite{MaHa99} when $g_{X_0^+(N)}\geq 3$.

\begin{thm}\cite[Lemma 2.1]{JKS04}
Let $X$ be a curve of genus $2$ over a perfect field $k$. If $X$ has
at least three $k$-rational points, then there exists a map
$X\rightarrow \mathbb{P}^1$ of degree $3$ which all is defined over
$k$.
\end{thm}
As an immediate consequence of the last theorem, we have:
\begin{prop}
$X_0^+(N)$ has infinitely many cubic points over $\Q$ for
 $$N\in \{42, 46, 52, 57, 67, 68, 69, 72, 73, 74, 77, 80, 91, 103, 107, 111, 121, 125,
143, 167, 191\}.$$
\end{prop}
\begin{proof}
Using \textsc{MAGMA} it can be easily checked that in this case the
genus $2$ hyperelliptic curve $X_0^+(N)$ has at least three
$\Q$-rational points.
\end{proof}
The remaining values of $N$ with $g_{X_0^+(N)}=2$ are $N=$ $62,
87$ and $98$.
\begin{prop}
For $N\in \{62, 87\}$, the set $\Gamma_3'(X_0^+(N),\Q)$ is not
finite.
\end{prop}
\begin{proof}
Consider $N=62$, an affine model of $X_0^+(62)$  is given by:
$$Y: y^2 = x^6 - 8x^5 + 26x^4 - 42x^3 + 29x^2 + 2x - 11.$$
Then $Y$ has two $\Q$-rational points ($(1:1:0)$ and $(1:-1:0)$)
which are the ``points at infinity", moreover the hyperelliptic
involution permutes them. Therefore, from~\cite[Lemma 2.2]{Jeo21} we
conclude that there is a $\Q$-rational degree $3$ mapping
$X_0^+(62)\rightarrow \mathbb{P}^1$ and consequently, $X_0^+(62)$
has infinitely many cubic points over $\mathbb{Q}$. A similar argument will work for
$N=87$ with the model $Y:y^2 = x^6 - 4x^5 + 12x^4 - 22x^3 + 32x^2 -
28x + 17$.
\end{proof}
\begin{lema}
The genus $2$ curve $X_0^+(98)$ has infinitely many cubic points
over $\Q$.
\end{lema}
\begin{proof}
A (affine) model of $X_0^+(98)$ is given by:
$$y^2=4x^5 - 15x^4 + 30x^3 - 35x^2 + 24x - 8.$$
Suppose $D$ is a degree $3$ effective $\Q$-rational divisor on a
curve of genus $2$. By Riemann-Roch theorem we have $\dim L(D)=2.$

Observe with $y=1$ in the model we get $0=(x^2-x+1)(x^3 -\frac{11}{4} x^2 +
\frac{15}{4}x - 9/4)$. Let $t_1, t_2, t_3$ be the roots of the equation
$x^3 - \frac{11}{4}x^2 + \frac{15}{4}x - \frac{9}{4}$. Then
$P_i:=(t_i,1)\in X_0^+(98)(K)$ for $1\leq i \leq 3$ (where $K$ is a
cubic extension of $\Q$ defined by the polynomial $t^3 -
\frac{11}{4}t^2 + \frac{15}{4}t - \frac{9}{4}$). Moreover, the divisor
$[P_1+P_2+P_3]$ is a $\Q$-rational effective divisor of degree $3$.
By Riemann-Roch we have $\dim L([P_1+P_2+P_3])=2$. Therefore, there
exists a $\Q$-rational function $f$ with exactly $3$ poles and
consequently there is a degree $3$ mapping $X_0^+(98)\rightarrow
\mathbb{P}^1$ defined over $\Q$. The result follows.
\end{proof}

Consider $X_0^+(N)$ hyperelliptic with $g_{X_0^+(N)}\geq 3$. By
\cite[\S 2.3]{Jeo21} in order to have infinite cubic points,
$W_3(X_0^+(N))$ must contain an elliptic curve with positive
$\mathbb{Q}$-rank.

Thus, by Cremona tables \cite{Cre} we obtain (because no elliptic
curve with $\Q$-rank $\geq 1$ for levels dividing $N$):

\begin{cor}
For $N \in \{60, 66, 85, 94, 104\}$, $\Gamma_3'(X_0^+(N),\Q)$ is a
finite set.
\end{cor}
%\begin{proof}
%For $N\in \{60, 66, 85, 94, 104\}$ there is no elliptic curve $E$
%with $\mathrm{cond(E)}\mid N$ and has positive $\Q$-rank. Hence for
%these values of $N$ the hyperelliptic curve $X_0^+(N)$ has finitely
%many cubic points.
%%For $N=92$, $92B1$ is the only elliptic curve with positive $\Q$-rank and conductor divides $92$. Suppose there is a $\Q$-rational degree $3$ mapping $X_0^+(92)\rightarrow 92B1$, hence we have a $\Q$-rational degree $6$ mapping $X_0(92)\rightarrow 92B1$.
%\end{proof}
\begin{prop}\label{Modular parametrization factors_hyperelliptic case}
$X_0^+(92)$ has infinitely many cubic points over $\Q$.
\end{prop}
\begin{proof}
The strong Weil modular parametrization $\phi: X_0(92)\rightarrow
92b$ has degree $6$ and $92b$ has $\Q$-rank $1$, and  $w_{92}$ acts
as $+1$ on $92b$, therefore we have a $\Q$-rational degree $3$ map
$X_0^+(92)\rightarrow 92b$ by Lemma \ref{lem1}.
\end{proof}

\subsection{Trigonal curves $X_0^+(N)$}
Suppose that
%$X_0^+(N)$ is trigonal, i.e.
$\mathrm{Gon}(X_0^+(N))=3$, the levels $N$ are:
$$\begin{tabular}{c|c}
$g_{X_0^+(N)}$ & $N$\\
\hline
3 & $58, 76, 86, 96, 97, 99, 100, 109, 113, 127, 128, 139, 149, 151, 169, 179, 239$\\
\hline
4 & 70, 82, 84, 88, 90, 93, 108, 115, 116, 117, 129, 135, \\
 & 137, 147, 155, 159, 161,
173, 199, 215, 251, 311\\
\hline
5 & 122, 146, 181, 227\\
\hline 6 & 164
\end{tabular}$$

If $g_{X_0^+(N)}=3$, then the projection from a $\Q$-rational cusp
defines a degree $3$ map $X_0^+(N)\rightarrow \mathbb{P}^1$ over
$\Q$ (cf.~\cite[Page 136]{HaS99}). On the other hand it is {known}
that every curve $C/K$ of genus $\geq 5$ with $\mathrm{Gon}(C)=3$ has a degree $3$
map $X\rightarrow \mathbb{P}^1$ over $K$ 
{(cf. \cite[Theorem 2.1]{NS}, \cite[Corollary 1.7]{HaS99})}
%(a similar proof of
%\cite[Lemma 5]{HaSi91}, see also \cite[Lemma 2.5]{BaMomose}).
%Thus we have the following result:
%\begin{prop}
%$X_0^+(N)$ has infinitely many cubic points for the following values of $N$:
%
%58, 76, 86, 96, 97, 99, 100, 109, 113, 127, 128, 139, 149, 151, 169, 179, 239,
%122, 146, 181, 227, 164.
%\end{prop}

Thus, we restrict to $\mathrm{Gon}(X_0^+(N))=3$ and
$g_{X_0^+(N)}=4$.

{It is well known that a non-hyperelliptic curve of genus $4$ lies
either on a quadratic cone or on a ruled surface (cf.~\cite[Page
136]{HaS99}), and by Petri's theorem a model of the curve can be
%by a Petri theorem we can
computed in $\mathbb{P}^{3}$
as the intersection of a degree 2 and a degree 3
homogenous equations. }
Following \cite[Page 131,p.136]{HaS99} {it can
be checked that for $N=159$ the curve $X_0^+(N)$ lies on a quadratic
cone over $\Q$ and for $N= 88, 93, 115, 116, 129, 137, 155, 215$ the
curve $X_0^+(N)$ lies on a ruled surface over $\Q$. On the other
hand for $N=70, 82, 84, 90, 108, 117, 135, 147, 161, 173, 199, 251, 311$ the
curve $X_0^+(N)$ lies on a ruled surface either over a quadratic
extension of $\Q$ or on a bi-quadratic extension of $\Q$. Hence in
these last levels the trigonal maps are not defined over $\Q$.} For
example, consider $X_0^+(70)$, the quadratic surface is given by $xz
- y^2 + 8yw - z^2 - 10zw - 9w^2$, after suitable coordinate change
this can be converted in the following equation $x^2 - y^2 - z^2 + 7w^2=
(x+y)(x-y)-(z+\sqrt{7}w)(z-\sqrt{7}w)$, and this surface is
isomorphic to the ruled surface $uv-st$ over
$\mathbb{Q}(\sqrt{7})$. See details in Appendix A for all
$X_0^+(N)$ trigonal with $g_{X_0^+(N)}=4$.

%Thus for $N=88, 93, 115, 116, 129, 137, 155, 159, 215,$ there is
%$\Q$-rational degree $3$ mapping $X_0^+(N)\rightarrow \mathbb{P}^1$
%from the arguments in \cite{HaS99}. Hence for these values of $N$,
%$X_0^+(N)$ has infinitely many cubic points.

 From the discussions so far we have:

\begin{thm}\label{values of N which are Trigonal over rationals}
Assume that $g_{X_0^+(N)}\geq 3$. Then $X_0^+(N)$ is trigonal over
$\Q$ if and only if $N$ is in the following list:

58, 76, 86, 88, 93, 96, 97, 99, 100, 109, 113, 115, 116, 122, 127, 128, 129, 137, 139, 146, 149, 151, 155, 159, 164, 169, 179, 181, 215, 227, 239.

In particular for such $N$, $\Gamma_3'(X_0^+(N),\mathbb{Q})$ is not
a finite set.
\end{thm}

Assume now, that $\mathrm{Gon}(X_0^+(N))=3$, but $X_0^+(N)$ does not admit a degree 3 map to the
projective line $\mathbb{P}^1$ over $\mathbb{Q}$.

 Hence in these cases
$X_0^+(N)$ contains infinitely many cubic points over $\Q$ when $W_3(X_0^+(N))$ contains a translation of the elliptic curve $E$
with positive $\Q$-rank \cite[Page 352]{Jeo21}.

\begin{prop}
For $N=70, 82, 84, 90, 108, 117, 135, 147, 161, 173, 199, 251, 311$,
the curve $X_0^+(N)$ has finitely many cubic points over $\Q$.
\end{prop}

\begin{proof}
For $N=70, 84, 90, 108, 147, 161, 173, 199, 251, 311$ there is no
elliptic curve $E$ of positive $\Q$-rank with $\mathrm{cond}(E)\mid
N$. Hence in these cases, $X_0^+(N)$ contains finitely many cubic points over $\mathbb{Q}$.

{ For $N=82,117$ and $135$, $X_0^+(N)$ is bielliptic and there are
elliptic curves of positive $\Q$-rank with $\mathrm{cond}(E)|N$. By
arguments in \cite[Page 353]{Jeo21}, if there is no $\Q$-rational
degree $3$ mapping $X_0^+(N)\rightarrow E$ where $E$ is an elliptic
curve of positive $\Q$-rank and $\mathrm{cond}(E)\mid N$, then
$W_3(X_0^+(N))$ has no translation of an elliptic curve with positive
$\Q$-rank.

In these cases only the following pairs $(82,82a)$,
$(117,117a)$ and $(135,135a)$ could appear.
If any of this pairs $(N,E)$ is admissible (i.e there is a $\mathbb{Q}$-rational degree $6$ mapping $X_0(N)\rightarrow E$), then the degree of the strong Weil parametrization of $E$ should divide $6$. For $82a, 117a$ and $135a$ the degrees of the strong Weil parametrization are $4,8$ and $16$ respectively. Thus no such pairs are admissible. The result follows.
%If should be admissible
%some of this pairs $(N,E)$ the degree of the map $X_0(N)\rightarrow
%E$ is 6 (by Lemma \ref{lem1} or \ref{lem2}) and should divide the
%degree of the strong Weil parametrization, that is 4, 8 and 16
%respectively, thus no such pairs are admissible. The result follows.
}

%For $N=82,117$ and $135$, $X_0^+(N)$ is bielliptic. only the following pairs $(82,82a)$,
%$(117,117a)$ and $(135,135a)$ could appear. If should be admissible
%some of this pairs $(N,E)$ the degree of the map $X_0(N)\rightarrow
%E$ is 6 (by Lemma \ref{lem1} or \ref{lem2}) and should divide the
%degree of the strong Weil parametrization, that is 4, 8 and 16
%respectively, thus no such pairs are admissible. By arguments in
%\cite{Jeo21} we have $W_3(X_0^+(N))$ has no translate of an elliptic
%curve with positive $\Q$-rank.
\end{proof}

\subsection{$X_0^+(N)$ bielliptic and not hyperelliptic and not
trigonal}

Suppose $X_0^+(N)$ is bielliptic but neither hyperelliptic nor trigonal.
%and not has a degree
%$3$ map to $\mathbb{P}^1$.
Following \cite[Page 353]{Jeo21} if
$\Gamma_3'(X_0^+(N),\mathbb{Q})$ is not finite, then $W_3(X_0^+(N))$
contains a translation of an elliptic curve $E$ with positive
$\Q$-rank, equivalently $(N,E)$ is an admissible pair.

The levels that remains to study are:

%78, 82, 105, 110, 117, 118, 120, 123, 124, 135, 136, 141, 142, 144,
%145, 171, 176, 188.

78, 105, 110, 118, 120, 123, 124, 136, 141, 142, 144,
145, 171, 176, 188.

\begin{prop} Suppose $X_0^+(N)$ is bielliptic and not hyperelliptic
and not trigonal, then the only admissible pair is $(124,124a)$, in
particular for all such curves, $\Gamma_3'(X_0^+(N),\mathbb{Q})$ is
not finite if and only if $N=124$.
\end{prop}
\begin{proof} For $N= 78, 105, 110, 120, 144, 188$ there is no
possible $(N,E)$ because there is no elliptic curve satisfying (iii)
in Lemma \ref{sieve} by Cremona tables \cite{Cre}. For $N= 118, 123,
124,
 136,$ $141$, $142$, $145$ the only possible admissible pairs $(N,E)$ are
 with $cond(E)=N$, but if they were admissible, we get a degree 6 map from
 $X_0(N)\rightarrow E$ and the degree of the strong Weil
 parametrization of $E$ (see Cremona table \cite{Cre} for such degrees) should divide 6, and no case happens except
 $(124,124a)$, and then by Lemma \ref{lem1} we conclude that the
 Weil parametrization of degree 6 factors through $X_0^+(124)$ because
 $w_{124}$ in $124a$ acts as $+1$. Finally, take $N=171, 176$, the pairs
 to study are $(171,171b)$, $(171,57a)$ and $(176,88a)$. The pair
 $(171,171b)$ we discard as the cases two lines before, because the strong Weil
 parametrization for $171b$ is 8. {We can apply Corollary \ref{lem2} with
 $w_{19}$ and $w_{11}$ respectively to obtain that $(171,57a)$ and
 $(176,88a)$ are not admissible.}
\end{proof}

\section*{Acknowledgments}

The second author thanks University Grants Commission, India for the financial support provided in the form of Research Fellowship to carry out this research work at IIT Hyderabad.
The second author would also like to thank the organizers and lecturers of the programme ``CMI-HIMR Summer School in Computational Number Theory". Some part of the paper was written during the second author's visit to the Universitat Aut\`{o}noma de Barcelona and the second author is grateful to the Department of Mathematics of Universitat Aut\`{o}noma de Barcelona for its support and hospitality.

\appendix
\section{A model for trigonal $X_0^+(N)$ with
$g_{X_0^+(N)}=4$}
For a detailed discussion on how to construct the models we refer the reader to \cite{Si} and \cite{Ga}.
\begin{small}

\begin{center}
\begin{tabular}{c|c}
Curve & Model of Petri for the Curve\\
\hline $X_0^+(70)$ &$x^2w - 7xw^2 - y^3 + 3y^2z + 2y^2w - 3yz^2 -
16yzw + 28yw^2 + z^3
    + 11z^2w - 19zw^2 - 27w^3,$\\
&$xz - y^2 + 8yw - z^2 - 10zw - 9w^2$\\
\hline $X_0^+(82)$ &$x^2w - 2 x y w - 5 x w^2 - y z^2 + 5 y z w + y
w^2 + 2 z^3 - 12 z^2 w +
    23 z w^2 - 9 w^3,$\\
&$x z - 3 x w - y^2 + 2 y z - 4 z^2 + 10 z w - 4 w^2$\\
\hline $X_0^+(84)$ &$x^2 w - 2 x y w - 5 x w^2 - y^2 z - y^2 w + 3 y
z^2 + 6 y z w + 5 y w^2 - 2 z^3
    - 6 z^2 w + 4 z w^2 + 4 w^3,$\\
&$x z - x w - y^2 + 2 y z + y w - 3 z^2 + w^2$\\
\hline $X_0^+(88)$ & $x^2 z - x y^2 - x y z - 2 x z^2 + y^3 + 6 y^2
z - 9 y^2 w - 8 y z^2 + 33 y w^2 +
    5 z^3 + 6 z^2 w - 12 z w^2 - 30 w^3,$\\
&$x w - y z + y w + z^2 - z w - 5 w^2$\\
\hline $X_0^+(90)$ &$x^2 w - 2 x y w - 3 x w^2 - y^2 z - y^2 w + 3 y
z^2 + 6 y z w + 3 y w^2 - 2 z^3
    - 5 z^2 w + z w^2,$\\
&$x z - x w - y^2 + 2 y z + y w - 3 z^2$\\
\hline
$X_0^+(93)$ & $x^2 z - x y^2 - x y z - 2 x z^2 + y^3 + 7 y^2 z - 11 y^2 w - 10 y z^2 + 7 y z w$\\
   & $+ 29 y w^2 + 6 z^3 + 2 z^2 w - 16 z w^2 - 21 w^3,$\\
&$x w - y z + y w + z^2 - 2 z w - 3 w^2$\\
\hline $X_0^+(108)$ &$x^2 w - 3 x w^2 - y^3 + 2 y^2 z - 8 y z w + 12
y w^2 - 2 z^3 + 12 z^2 w -
    22 z w^2 + 5 w^3,$\\
&$x z - y^2 + 4 y w - 6 z w - w^2$\\
\hline $X_0^+(115)$ & $x^2 z - x y^2 - x y z - 2 x z^2 + y^3 + 5 y^2
z - 9 y^2 w - 4 y z^2 - 6 y z w +
    29 y w^2 + 2 z^3 + 5 z^2 w - 22 w^3,$\\
&$x w - y z + y w + z^2 - 4 w^2$\\
\hline $X_0^+(116)$ &$x^2 z - x y^2 - 2 x z^2 + 4 y^2 z + 2 y^2 w -
6 y z^2 - 8 y z w + 3 y w^2 +
    4 z^3 + 9 z^2 w - 4 z w^2 - 4 w^3,$\\
&$x w - y z + z^2 - 3 w^2$\\
\hline $X_0^+(117)$ &$x^2 w - x y w - 5 x w^2 - y^2 z + y^2 w + y
z^2 + y z w + y w^2 - z^3 + 2 z w^2
    + 4 w^3,$\\
&$x z - y^2 + y z + y w - 3 z^2 + 2 z w - 4 w^2$\\
\hline $X_0^+(129)$ &$x^2 z - x y^2 - 2 x z^2 + 5 y^2 z - 7 y z^2 -
3 y z w + 3 y w^2 + 4 z^3 +
    3 z^2 w - 3 z w^2 - w^3,$\\
&$x w - y z + z^2 - z w - w^2$\\
\hline $X_0^+(135)$ &$x^2 w - 2 x y w - 3 x w^2 - y^3 + 3 y^2 z + 2
y^2 w - 3 y z^2 + 2 y w^2 + z^3 +
    w^3,$\\
&$x z - 2 x w - y^2 + 2 y z + 3 y w - 2 z^2 - z w$\\
\hline $X_0^+(137)$ &$x^2 z - x y^2 - x z^2 + 3 y^2 z + 2 y^2 w - 6
y z^2 - y z w - 3 y w^2 + 3 z^3 +
    2 z^2 w - z w^2 + 2 w^3,$\\
&$x w - y z + z^2 - z w - w^2$\\
\hline $X_0^+(147)$ &$x^2 w - x y w - 6 x w^2 - y^2 z + y z^2 + 2 y
w^2 - z^3 + z^2 w + 3 z w^2 +
    7 w^3,$\\
&$x z - x w - y^2 + y z - 2 z^2 + z w + w^2$\\
\hline $X_0^+(155)$ &$x^2 z - x y^2 - x y z - x z^2 + y^3 + 3 y^2 z
- 5 y^2 w - 2 y z^2 + 2 y z w +
    7 y w^2 + z^3 - 2 z w^2 - 3 w^3,$\\
&$x w - y z + y w - 2 w^2$\\
\hline $X_0^+(159)$ &$x^2 z - x y^2 + x y z - 3 x z^2 + 2 y^2 z +
y^2 w - 8 y z w + 3 y w^2 + 7 z^2 w
    - z w^2 - 2 w^3,$\\
&$x w - y w - z^2 + 2 z w - 2 w^2$\\
\hline
$X_0^+(161)$ &$x^2 w - 5 x w^2 - y^2 z + y z^2 + 2 y w^2 - 3 z^2 w + 9 z w^2 - 4 w^3,$\\
&$x z - x w - y^2 + 3 y w - z^2 + z w - 3 w^2$\\
\hline $X_0^+(173)$ &$x^2 w - x y w + 6 x w^2 - 2 y^2 w - y z^2 + 4
y z w + 6 y w^2 + 4 z^2 w -
    17 z w^2 - 6 w^3,$\\
&$x z + 2 x w - y^2 + y z + 3 y w - 6 z w - 3 w^2$\\
\hline $X_0^+(199)$ &$x^2 w + 2 x y w + x w^2 - y^3 - y^2 z + 2 y^2
w + y z^2 - 5 y z w + 3 z w^2 -
    5 w^3,$\\
&$x z + 2 x w - y^2 - 2 y z + 3 y w - 4 w^2$\\
\hline $X_0^+(215)$ &$x^2 z - x y^2 - x y z - x z^2 + y^3 + 2 y^2 z
- 3 y^2 w - 2 y z w + 5 y w^2 +
    z^3 - z^2 w + z w^2 - 2 w^3,$\\
&$x w - y z + y w + z w - 2 w^2$\\
\hline
$X_0^+(251)$ &$x^2 w - 5 x w^2 - y^2 z - y^2 w + y z^2 + y w^2 + z^2 w - z w^2 + 4 w^3,$\\
&$x z - 2 x w - y^2 + y w + w^2$\\
\hline
$X_0^+(311)$ &$x^2 w - x y w - y^3 + y^2 z + 2 y^2 w - y z^2 - 2 y z w - y w^2 + z^2 w,$\\
&$x z - x w - y^2 + y z + 2 y w - z^2 - 2 z w$\\
\hline
\end{tabular}
\end{center}

\end{small}

\begin{small}

\begin{center}
\begin{tabular}{c|c}
Curve & quadratic surface\\
\hline $X_0^+(70)$
& Diagonal form: $x^2 - y^2 - z^2 + 7 w^2$\\
&lies on the ruled defined over $\Q(\sqrt{7})$ \\
\hline $X_0^+(82)$
& Diagonal form: $3 x^2 - 12 y^2 - 4 z^2 - w^2$\\
&lies on a ruled surface over $\Q(\sqrt{-1})$  \\
\hline $X_0^+(84)$
& Diagonal form: $2 x^2 - 6 y^2 - 3 z^2 + w^2$\\
 &lies on a ruled surface over $\Q(\sqrt{3})$\\
\hline $X_0^+(88)$
& Diagonal form: $5 x^2 + 5 y^2 - 5 z^2 - 5 w^2$\\
  &lies on a ruled surface over $\Q$ \\
\hline $X_0^+(90)$
& Diagonal form: $2 x^2 - 6 y^2 - 3 z^2 - 3 w^2$\\
 & lies on a ruled surface over $\Q(\sqrt{3},\sqrt{-1})$\\
\hline $X_0^+(93)$
&Diagonal form: $4 x^2 + 3 y^2 - 4 z^2 - 3 w^2$ \\
&lies on a ruled surface over $\Q$ \\
\hline $X_0^+(108)$
&Diagonal form: $-x^2 - y^2 + z^2 + 3 w^2$ \\
 & lies on a ruled surface over $\Q(\sqrt{3})$\\
\hline $X_0^+(115)$
&Diagonal form: $3 x^2 + 4 y^2 - 3 z^2 - 4 w^2$\\
 &lies on a ruled surface over $\Q$ \\
\hline $X_0^+(116)$
& Diagonal form: $3 x^2 - y^2 + z^2 - 3 w^2$\\
 &lies on a ruled surface over $\Q$ \\
\hline $X_0^+(117)$
& Diagonal form: $11 x^2 - 33 y^2 - 3 z^2 - 15 w^2$\\
 & lies on a ruled surface over $\Q(\sqrt{3},\sqrt{-5})$\\
\hline $X_0^+(129)$
& Diagonal form: $x^2 - 5 y^2 + 5 z^2 - w^2$\\
 &lies on a ruled surface over $\Q$ \\
\hline $X_0^+(135)$
&Diagonal form: $x^2 - 2 y^2 - 2 z^2 + 9 w^2$\\
&lies on a ruled surface over $\Q(\sqrt{2})$ \\
\hline $X_0^+(137)$
&Diagonal form: $x^2 - 5 y^2 + 5 z^2 - w^2$\\
 &lies on a ruled surface over $\Q$ \\
\hline $X_0^+(147)$
&Diagonal form: $7 x^2 - 14 y^2 - 2 z^2 + w^2$\\
 &lies on a ruled surface over $\Q(\sqrt{2})$ \\
\hline $X_0^+(155)$
& Diagonal form: $2 x^2 + 2 y^2 - 2 z^2 - 2 w^2$\\
 &lies on a ruled surface over $\Q$ \\
\hline $X_0^+(159)$
& Diagonal form: $2 y^2 - z^2 - 2 w^2$\\
 &lies on a quadratic cone over $\Q$ \\
\hline $X_0^+(161)$
& Diagonal form: $x^2 - y^2 - z^2 - 3 w^2$\\
&lies on a ruled surface over $\Q(\sqrt{-3})$ \\
\hline $X_0^+(173)$
&Diagonal form: $-x^2 - 3 y^2 + 3 z^2 + 37 w^2$\\
&lies on a ruled surface over $\Q(\sqrt{37})$ \\
\hline $X_0^+(199)$
& Diagonal form: $-x^2 - y^2 + z^2 + 33 w^2$\\
&lies on a ruled surface over $\Q(\sqrt{33})$ \\
\hline $X_0^+(215)$
&Diagonal form: $x^2 + 2 y^2 - z^2 - 2 w^2$\\
 &lies on a ruled surface over $\Q$ \\
\hline $X_0^+(251)$
&Diagonal form: $-x^2 - y^2 + z^2 + 5 w^2$\\
&lies on a ruled surface over $\Q(\sqrt{5})$ \\
\hline $X_0^+(311)$
&Diagonal form: $3 x^2 - 3 y^2 - z^2 - 3 w^2$\\
 &lies on a ruled surface over $\Q(\sqrt{-3})$ \\
\hline
\end{tabular}
\end{center}

\end{small}

\section{The sieves to reduce to a finite set of $N$ to consider}

{Here we consider the levels $N$ that does not appear
in Theorem \ref{known values of N}}. A similar Ogg's classical
argument as in the proof of~\cite[Lemma 3.2]{HS99} one obtains, if
$N\geq 624$, there does not exist a $\Q$-rational degree $6$ mapping
$X_0(N)\rightarrow E$ for any $E$. Consequently, no degree $3$ map
$X_0^+(N)\rightarrow E$ over $\mathbb{Q}$ for $N\geq 624$.

Now by (i) in Lemma \ref{sieve} we can discard
the existence of such degree 3 map for $N$ in:

\begin{small}

252, 260, 264, 272, 276, 280, 288, 290, 294, 296, 300, 304, 306,
308, 310, 312, 315, 316, 318, 320, 322, 324, 328, 330, 332, 336,
340, 342, 344, 345, 348, 350, 352, 354, 356, 357, 360, 364, 366,
368, 370, 372, 374-376, 378, 380, 382, 384, 385, 386, 388,
390, 392, 394, 396, 398-400, 402, 404-406, 408, 410,
412, 414, 416, 418, 420, 422-426, 428-430,
432, 434-436, 438, 440-442, 444, 446, 448, 450, 452-456, 458-460, 462, 464-466, 468, 470-472, 474-478, 480, 482-486,
488-490, 492, 494-498, 500-502, 504-508, 510-520, 522, 524-528, 530-540, 542-546,
548-556, 558-562, 564-623.

\end{small}
%
%By Theorem we know that $\Gamma_3^\prime(X_0^+(N), \Q)$ is infinite
%if and only if $W_3(X_0^+(N))$ contains a translation of an elliptic
%curve with positive $\Q$-rank, moreover the conductor of such
%elliptic curves will divide $N$. Thus, by Cremona's table we can
%eliminate the following $N$ for which there exists no elliptic curve
%of positive $\Q$-rank with conductor divides $N$:
By Lemma \ref{sieve} (iii) we can discard all pairs $(N,E)$ for the
following $N$:

\begin{small}

126, 132, 133, 134, 140, 150, 157, 165, 168, 177, 180, 186, 187,
193, 194, 206, 211, 213, 217, 221, 223, 230, 233, 240, 241, 247,
250, 253, 255, 257, 261, 263, 266, 268, 271, 279, 281, 283, 287,
292, 293, 295, 299, 307, 313, 317, 319, 321, 323, 329, 334, 337,
341, 343, 349, 353, 355, 358, 365, 367, 379, 383, 391, 397, 401,
403, 409, 411, 413, 417, 419, 421, 439, 447, 449, 457, 461, 463,
479, 487, 491, 499, 509, 521, 523, 529, 541, 547.

%146,\footnote{{\color{blue}{$N=146$ is in Theorem 2, may be deleted}}}

\end{small}

%
%Let $f: X_0(N)\rightarrow E$ be a $\Q$-rational map of degree $6$,
%where $E$ is an elliptic curve with positive $\Q$ rank. If
%$\mathrm{Cond}(E)=N$, then $f$ factors through the modular
%parametrization $\phi: X_0(N)\rightarrow E^\prime$, where $E^\prime$
%is the strong Weil curve. For such $N$, we must have
%$\mathrm{deg}(\phi)\leq 6$. {\color{red}Therefore, we can eliminate
%the following values of $N$ for which $\mathrm{deg}(\phi)>6$ and
%there exists no elliptic curve of positive $\Q$-rank whose conductor
%is a proper divisor of $N$:}

By the use of (iii) and (v) in Lemma \ref{sieve} we can discard $N$ in
the list:

\begin{small}

102, 112, 138, 152, 153, 156, 160, 170, 175, 189, 190, 192, 197,
200, 201, 203, 205, 207, 208, 209, 210, 214, 216, 218, 219, 220, 225,
226, 229, 235, 238, 245, 254, 274, 275, 277, 278, 289, 291, 298, 302,
309, 314, 327, 331, 335, 338, 339, 346, 347, 359, 361, 362, 373,
377, 381, 389, 431, 433, 437, 443, 451, 467, 469, 493, 503, 557,
563.

\end{small}

{{For $N$ in the table bellow, using (v) in Lemma
\ref{sieve} we can eliminate all $(N,E)$ with $cond(E)=N$, the remaining pairs $(N,E)$ where $cond(E)|N$ and $cond(E)\ne N$ ($rank_{\mathbb{Q}}(E)\geq 1$) can be eliminated by
(ii) of Lemma \ref{sieve}
 %we can discard}} all the pairs $(N,E)$
%with $cond(E)|N$ (always $rank_{\mathbb{Q}}(E)\geq 1$)
i.e by computing}}
$\mathbb{F}_{p^r}$-points on $X_0(N)$ with $p\nmid N$ in the first
two columns and the last one for $X_0^+(N)$ instead of $X_0(N)$. {Thus we can discard all the levels $N$ appearing in the table below.}

\begin{small}
\begin{center}
\begin{tabular}{ |c|c|c| }
\hline
$N$ & $E$  & $p^r$ \\
\hline
$148$ & $37a$ & $3^2$  \\
$154$ & $77a$ & $3^2$  \\
$184$ &$92b$ & $3^2$  \\
$198$ &$99a$ & $5^2$  \\
$204$ &$102a$ & $5^2$  \\
$212$ &$53a$ & $3^2$  \\
$212$ &$106a$ & $5^2$  \\
$224$ &$112a$ & $3^2$  \\
$228$ &$57a$ & $5^2$  \\
$232$ &$58a$ & $3^2$  \\
$234$ &$117a$ & $5^2$ \\
$242$ & $121b$ & $5^2$ \\
$246$ &$82a$ &$7^2$  \\
$246$ &$123a$ &$5^2$  \\
$246$ &$123b$ &$7^2$  \\
$256$ &$128a$ &$3^2$ \\
$259$ &$37a$ &$3^2$ \\
$265$ &$53a$ &$3^2$ \\
$270$ &$135a$ &$7^2$ \\
%$273$ &$91A1$ &$2^2$ & &\\
%$273$ &$91B1$ & & &\\
$285$ &$57a$ &$2^2$\\
$286$ &$143a$ &$3^2$\\
\hline
\end{tabular}
%&
\begin{tabular}{ |c|c|c| }
\hline
$N$ & $E$  & $p^r$  \\
\hline
$297$ &$99a$ &$5^2$\\
$301$ &$43a$ &$5^2$\\
$325$ &$65a$ &$3^2$\\
$326$ &$163a$ &$3^2$\\
$333$ &$37a$ &$5^2$ \\
$351$ &$117a$ &$2^2$ \\
$363$ &$121a$ &$5^2$\\
$369$ &$123a$ &$2^2$ \\
$369$ &$123b$ &$7^2$ \\
$371$ &$53a$ &$3^2$ \\
$387$ &$43a$ &$2^2$ \\
$387$ &$129a$ &$2^4$ \\
$393$ &$131a$ &$5^2$ \\
$407$ &$37a$ &$3^2$ \\
$415$ &$83a$ &$3^2$ \\
$427$ &$61a$ &$3^2$ \\
$445$ &$89a$ &$2^2$ \\
$473$ &$43a$ &$2^2$ \\
$481$ &$37a$ &$2^2$ \\
\hline
\end{tabular}
%&
\begin{tabular}{ |c|c|c|}
\hline
$N$ & $E$  & $p^r$ \\
\hline
%$237$ &$79a$ &  ? \\
$244$ &$61a$ & $3^2$  \\
$244$ &$122a$ & $3^2$  \\
$248$ &$124a$ & $5^2$  \\
%$267$ &$89a$ &  ? \\
$273$ &$91a$ &$2$   \\
$273$ &$91b$ & $2^2$  \\
$282$&$141a$&$5^2$\\
$282$&$141d$&$7^2$\\
$305$ &$61a$ & $7$  \\
$395$ &$79a$ &  $2^2$ \\
\hline
\end{tabular}

\end{center}
\end{small}

%{\color{blue}{From the $N$ of the table above, we can discard all such conductors.}}

%In the above table the first column denote the value of $N$
%corresponding to the curve $X_0^+(N)$ and the second column denote
%the elliptic curves $E$ of positive $\Q$-rank such that
%$\mathrm{cond}(E)\mid N$ and $\mathrm{cond}(E)< N$. Note that for
%the above values of $N$ all the elliptic curves of positive degree
%with conductor $N$ has strong Weil degree $>6$. Hence for the above
%values of $N$ there does not exist any $\Q$-rational degree $3$
%ampping $X_0^+(N)\rightarrow E$ where $E$ is an elliptic curve of
%positive $\Q$-rank such that $\mathrm{cond}(E)\mid N$.
%
%
%
%On the other hand if there is a $\Q$-rational degree $3$ map
%$X_0^+(N)\rightarrow E$ and $p$ is a prime with $p\nmid N$, then we
%must have $|\tilde{X}_0^+(N)(\F_{p^r})|\leq 3|\tilde{E}(\F_{p^r})|$
%for all $r\in \N$. Where $|\tilde{X}_0^+(N)(\F_{p^r})|$ denotes the
%number of $\F_{p^r}$-rational points on $X_0^+(N)$.
%
%
%{\color{red}Using MAGMA we can compute
%$|\tilde{X}_0^+(N)(\F_{p^r})|$ and $\tilde{E}(\F_{p^r})$ and
%eliminate the following values of $N$:}
%
%%237, 244, 248, 267, 273, 305, 395.
%\begin{center}
%
%

%\begin{prop}

By (iv) in Lemma \ref{sieve} we can discard $N=222,
262, 284, 303$.
% the modular curve $X_0^+(N)$ has
%finitely many cubic points.
%\end{prop}
%
%\begin{proof}
%By Theorem~\ref{N not in known set the it has map to elliptic
%curve}, we know that for the values of $N$, if $X_0^+(N)$ is
%infinite, then there exists a $\Q$-rational degree $3$ map
%$X_0^+(N)\rightarrow E$ where $E$ is an elliptic curve with positive
%rank over $\Q$.
%
%Assume $N=262$, suppose there is a $\Q$-rational degree $3$ map
%$X_0^+(262)\rightarrow E$ where $E$ is an elliptic curve with
%positive rank over $\Q$.
%
%We know that there is a degree $2$ mapping $X_0^+(262)\rightarrow
%X_0^+(262)/w_2$, where genus($X_0^+(262)/w_2$) = $4$. By
%Theorem~\ref{Castelnuovo's Inequality} we must have
%$$\mathrm{genus}(X_0^+(262))\leq 3.1+2.4+ 2.1$$
%which is a contradiction since $\mathrm{genus}(X_0^+(262))=15$.
%Therefore, $X_0^+(262)$ has finitely many cubic points.
%
%Using a similar argument we can conclude that $X_0^+(222),
%X_0^+(284), X_0^+(303)$ have finitely many cubic points.
%\end{proof}

\bibliographystyle{alpha}

\noindent{Francesc Bars Cortina}\\
{Departament Matem\`atiques, Edif. C, Universitat Aut\`onoma de Barcelona\\
08193 Bellaterra, Catalonia}\\
{francesc@mat.uab.cat}

\vspace{1cm}

\noindent{Tarun Dalal}\\
 {Department of Mathematics, Indian Institute of Technology Hyderabad,\\
  Kandi, Sangareddy 502285, INDIA.}\\
 {ma17resch11005@iith.ac.in}

 \end{document}